\documentclass[11pt,letterpaper,reqno]{amsart}

\usepackage{amssymb}
\usepackage{amsmath}
\usepackage{amsthm}
\usepackage{bbm}
\usepackage{doi}
\usepackage{enumitem}

\usepackage{bookmark}
\usepackage{hyperref}
\hypersetup{pdfstartview={FitH}}

\numberwithin{equation}{section}

\newtheorem{theorem}{Theorem}
\newtheorem{lemma}[theorem]{Lemma}

\newtheorem{corollary}[theorem]{Corollary}

\renewcommand{\leq}{\leqslant}
\renewcommand{\geq}{\geqslant}

\begin{document}

\title[On eventually greedy best underapproximations by Egyptian fractions]{On eventually greedy best underapproximations by Egyptian fractions}

\author[Vjekoslav Kova\v{c}]{Vjekoslav Kova\v{c}}

\address{Department of Mathematics, Faculty of Science, University of Zagreb, Bijeni\v{c}ka cesta 30, 10000 Zagreb, Croatia}

\email{vjekovac@math.hr}




\begin{abstract}
Erd\H{o}s and Graham found it conceivable that the best $n$-term Egyptian underapproximation of almost every positive number for sufficiently large $n$ gets constructed in a greedy manner, i.e., from the best $(n-1)$-term Egyptian underapproximation. We show that the opposite is true: the set of real numbers with this property has Lebesgue measure zero.
\end{abstract}

\maketitle


\section{Introduction}
Ancient Egyptians preferred to write positive rational numbers as sums of distinct unit fractions, i.e., fractions with numerators $1$ and mutually different positive denominators. Modern-day interest in such representations might have been revived by Sylvester \cite{Syl80}, who also considered expansions of real numbers as infinite series of distinct unit fractions.
The ``Egyptian fractions'' quickly lead to numerous problems that are easy to formulate, but difficult to approach, so they have been one of the favourite occupations of the late Paul Erd\H{o}s \cite[\S 4]{EG80}.
Excellent survey papers on them were written by Graham \cite{Gra13} and by Bloom and Elsholtz \cite{BE22}. The topic is still very active and several breakthroughs have been obtained very recently \cite{LS24,Con24}.

For a positive integer $n$, we say that a rational number $q\in[0,\infty)$ is an \emph{$n$-term Egyptian underapproximation} of a real number $x\in(0,\infty)$ if
\begin{equation}\label{eq:under}
q = \sum_{k=1}^{n} \frac{1}{m_k} < x
\end{equation}
for some positive integers $m_1<m_2<\cdots<m_n$.
Conventions vary throughout the literature; for instance Nathanson \cite{Nat23} defines an Egyptian underapproximation to be a tuple of denominators $(m_1,m_2,\ldots,m_n)$, rather than a number $q$. We prefer the above definition, but we need to be aware that some rational numbers $q$ have many representations \eqref{eq:under} as sums of distinct unit fractions. It is also convenient to declare that $0$ is the unique $0$-term underapproximation of every number $x>0$.
It is easy to see that the largest $n$-term Egyptian underapproximation of any given $x\in(0,\infty)$ exists (see \cite[Theorem 3]{Nat23}); it will also be called \emph{the best $n$-term Egyptian underapproximation} of $x$. 

An effective way of approximating a number $x>0$ from below by sums of distinct unit fractions is to use a greedy algorithm.
We  define the \emph{greedy $n$-term Egyptian underapproximation} of $x>0$ as the number $\sum_{k=1}^{n} 1/m_k$, where $(m_k)_{k=1}^{\infty}$ is now the sequence of positive integers defined recursively as
\[ m_n := \Big\lfloor \Big(x - \sum_{k=1}^{n-1} \frac{1}{m_k}\Big)^{-1} \Big\rfloor + 1 \]
and the empty sum $\sum_{k=1}^{0}$ is interpreted as $0$.
It is easy to construct numbers $x$ for which these greedy underapproximations are not optimal; the fraction $x=11/24=0.45833\ldots$ being just one example. Namely, its greedy two-term Egyptian underapproximation is $1/3+1/9=0.44\ldots$, but its best two-term Egyptian underapproximation is larger, being $1/4+1/5=0.45$.
For this reason a very nice observation was that the best and the greedy $n$-term Egyptian underapproximations of the number $x=1$ coincide for every positive integer $n$; it was proved by Curtiss \cite{Cur22} and Takenouchi \cite{Tak21}. Several alternative proofs have appeared since; an elegant one has been given by Soundararajan \cite{Sou05}.
Erd\H{o}s \cite{Erd50} observed that this remains to hold whenever $x=1/b$ itself is a unit fraction.
Nathanson \cite[Theorem 5]{Nat23} showed that the same property is retained by rationals $x=a/b$ such that $a$ divides $b+1$, while Chu \cite[Theorem 1.12]{Chu23} extended it also to proper reduced fractions $x=a/b$ such that $b$ is odd and $l=2$ is the smallest positive integer such that $a$ divides $b+l$.

A weaker notion of ``eventually greedy'' underapproximations was suggested by Erd\H{o}s and Graham \cite[p.~31]{EG80}. 
We say that $x\in(0,\infty)$ has \emph{eventually greedy best Egyptian underapproximations} if there exist a strictly increasing sequence of positive integers $(m_k)_{k=1}^{\infty}$ and an integer $n_0\geq0$ such that for every $n\geq n_0$ the sum $\sum_{k=1}^{n} 1/m_k$ equals the best $n$-term Egyptian underapproximation of $x$.
In words, for sufficiently large $n$, an upgrade from the best $(n-1)$-term to the best $n$-term underapproximation of $x$ is performed by simply adding another, previously non-existing, unit fraction.
Erd\H{o}s and Graham, in their 1980 monograph \cite[p.~31]{EG80}, claimed that every positive rational has eventually greedy best Egyptian underapproximations, but gave no proof or a reference. Then they commented at the bottom of page 31:
\begin{quote}
\emph{It is not difficult to construct irrationals for which the result fails. Conceivably, however, it holds for almost all reals.}
\end{quote}
Whether or not almost every positive real number has eventually greedy best Egyptian underapproximations was formulated as Problem \#206 on Bloom's website \emph{Erd\H{o}s problems} \cite{EP}.
Moreover, Nathanson \cite[\S~Open problems, (4)]{Nat23} recently questioned the mere existence of such irrational numbers, since Erd\H{o}s and Graham provided no proof or a reference for that claim either.
The following theorem answers these questions.

\begin{theorem}\label{thm:main}
The set of positive real numbers with eventually greedy best Egyptian underapproximations has Lebesgue measure zero.
\end{theorem}

Thus, the question stemming from Erd\H{o}s and Graham's comment has a negative answer that is even at the other extreme of possibilities. Since the set of rational, and even algebraic, numbers is negligible in the measure-theoretic sense (being only countable), this also gives a positive, but non-constructive, answer to Nathanson's question. 

\begin{corollary}\label{cor:main}
There exists a transcendental real number that does not have eventually greedy best Egyptian underapproximations.
\end{corollary}

Non-constructive results are often obtained by the probabilistic method, which is how our approach could also be phrased. In particular, the proof below leaves no clue on how to find a particular transcendental number to which Corollary \ref{cor:main} applies.

\smallskip
A recursive construction will reduce Theorem \ref{thm:main} to the study of two-term underapproximations only.
We will need to find a positive share, independent of $i$ (such as $1\textup{\textperthousand}=1/1000$), of the numbers in the interval $(1/i,1/(i-1)]$ that have a more efficient two-term Egyptian underapproximation than the greedy one. Note that the greedy two-term underapproximation of any number from this interval is always of the form $1/i + 1/j$ for some $j\geq (i-1)i+1$.

\begin{lemma}\label{lm:main}
For every integer $i\geq1000$, at least $1\textup{\textperthousand}$ of the numbers in the interval\linebreak $(1/i,1/(i-1)]$ have non-greedy best two-term Egyptian underapproximations.
\end{lemma}

The numbers for which the greedy two-stage process is not optimal have been studied by Nathanson \cite{Nat23} and Chu \cite{Chu23}, but here we need to quantify the Lebesgue measure of the set that they make. Luckily, we do not need to be anyhow precise, as the result will ``bootstrap itself'' when we start considering underapproximations with more unit fractions.

\smallskip
Erd\H{o}s and Graham also wrote the following regarding the eventually greedy best Egyptian underapproximation property \cite[p.~31]{EG80}:
\begin{quote}
\emph{An attractive conjecture is that this also holds for any algebraic number as well.}
\end{quote}
Representations of algebraic numbers as greedy series of distinct unit fractions and their algorithmic aspects have been studied by Stratemeyer \cite{Str30} and Salzer \cite{Sal47}.
Finally, some $33$ years after \cite{EG80}, Graham \cite[295--296]{Gra13} implicitly retracted the claim that all rationals have eventually greedy Egyptian underaproximations. He first mollified it to:
\begin{quote}
\emph{Perhaps it is true that for any rational it does hold eventually.}
\end{quote}
and then stated it explicitly as yet another open question. This was also asked by Nathanson \cite[\S~Open problems, (4)]{Nat23}.
Here we do not make any progress towards these interesting open problems.

\smallskip
The following two sections are respectively dedicated to the proofs of Lemma \ref{lm:main} and Theorem \ref{thm:main}.
For real sequences $(x_i)$ and $(y_i)$ we write $x_i = \Theta(y_i)$ if both $|x_i|\leq C|y_i|$ and $|y_i|\leq C|x_i|$ hold for some unimportant constant $C\in(0,\infty)$ and every index $i$.
This notation will only be used in heuristic considerations, as otherwise we will prefer to be explicit.


\section{Proof of Lemma \ref{lm:main}}

Fix an integer $i\geq1000$. Consider numbers of the form
\begin{equation}\label{eq:twotermE}
\frac{1}{i+1} + \frac{1}{i(i+1)/2 + k}
\end{equation}
for integers $0\leq k\leq \lfloor i(i+1)/10\rfloor$.
If we define
\begin{equation}\label{eq:sequence}
x_k := \Big(\frac{1}{i+1} + \frac{1}{i(i+1)/2 + k} - \frac{1}{i}\Big)^{-1}
= i(i+1)\frac{i(i+1) + 2k}{i(i+1) - 2k},
\end{equation}
then the number \eqref{eq:twotermE} can be rewritten as
\[ \frac{1}{i} + \frac{1}{x_k} \]
and it falls into
\begin{equation}\label{eq:twotermI}
\Big(\frac{1}{i}+\frac{1}{j_k},\frac{1}{i}+\frac{1}{j_k-1}\Big], \quad j_k = \lfloor x_k\rfloor+1.
\end{equation}
Clearly,
\[ j_k > x_k \geq i(i+1) > (i-1)i+1, \] 
and from this we see that each interval \eqref{eq:twotermI} lies fully inside $(1/i,1/(i-1)]$.
We also have
\[ x_{k+1} - x_k = \frac{4i^2 (i+1)^2}{(i(i+1)-2k)(i(i+1)-2k-2)} > 1. \]
Thus, different values of $k$ lead to different integers $j_k$, so the obtained subintervals \eqref{eq:twotermI} are mutually different. 
The point \eqref{eq:twotermE} splits \eqref{eq:twotermI} into its left and right parts, the right one being
\begin{equation}\label{eq:twotermR}
\Big(\frac{1}{i} + \frac{1}{x_k}, \frac{1}{i}+\frac{1}{\lfloor x_k\rfloor}\Big].
\end{equation}
Every point from \eqref{eq:twotermR} has its best two-term Egyptian underapproximation at least equal to \eqref{eq:twotermE}, which is strictly larger than a greedy one, $1/i+1/j_k$.

Since there are $\Theta(i^2)$ considered intervals \eqref{eq:twotermI} and their lengths are easily seen to be $\Theta(j_k^{-2})=\Theta(x_k^{-2})=\Theta(i^{-4})$, the main idea is now to show that there are $\Theta(i^2)$ points \eqref{eq:twotermE} that fall reasonably far from the right endpoint of the corresponding interval \eqref{eq:twotermI}.
That way there will be $\Theta(i^2)$ intervals \eqref{eq:twotermR} with lenght $\Theta(i^{-4})$, so their total length will be $\Theta(i^{-2})$. This is a positive fraction of $1/((i-1)i)$, as desired. We remark that Chu \cite[\S3]{Chu23} has already used the idea of comparing the non-greedy two-term underapproximation $1/(i+1)+\cdots$ to the greedy one, namely $1/i+\cdots$, but gave no estimates on the length of the set for which the former underapproximation is superior.

The task is now reduced to finding $\Theta(i^2)$ terms of the above finite sequence $(x_k)$, such that $x_k$ is reasonably far from $\lfloor x_k\rfloor$.
In order to achieve this, one is tempted to study equidistribution properties of
\[ x_k = i(i+1) + 4k + \frac{8k^2}{i(i+1)-2k} \]
modulo $1$, but this leads to unnecessary complications, since $(x_k)$ is not quite a polynomial sequence; it only emulates the quadratic behaviour of $8k^2/(i(i+1))$ modulo $1$.
Luckily, the difference sequence $(x_{k+1}-x_{k})$ has a more stable dynamics and we will be content if, for many indices $k$, at least one of the terms $x_k$ and $x_{k+1}$ is far from $0$ modulo $1$. One might view at this trick as a cheap substitute for van der Corput's lemma.
More information about the relevant concepts from the theory of equidistribution can be found in \cite{KN74}, but we will not need them for the rest of the proof. These tools and the related Diophantine approximations \cite[Chapter 8]{Duj21} could become useful if one desires to obtain a quantitatively sharper version of Lemma \ref{lm:main}. 

Let us now make the arguments rigorous and finish the proof. 
Define 
\begin{equation}\label{eq:setL}
\mathcal{L} := \Big[\frac{i(i+1)}{100},\frac{3i(i+1)}{200}\Big] \cap \mathbb{Z}
\end{equation}
and observe that its cardinality is at least $i^2/200$.
For every $l\in\mathcal{L}$ the formula
\[ x_{2l+1} - x_{2l} = \frac{4i^2(i+1)^2}{(i(i+1)-4l)(i(i+1)-4l-2)} \]
guarantees
\begin{equation}\label{eq:difference}
4 + \frac{1}{3} \leq x_{2l+1} - x_{2l} \leq 4 + \frac{2}{3}.
\end{equation}
Consequently,
\[ x_{2l} - \lfloor x_{2l}\rfloor \geq \frac{1}{3} \quad\text{or}\quad x_{2l+1} - \lfloor x_{2l+1}\rfloor \geq \frac{1}{3}, \]
since otherwise $x_{2l+1} - x_{2l}$ would differ from an integer by less than $1/3$, which would contradict \eqref{eq:difference}.
Also note that, by the definitions \eqref{eq:sequence} and \eqref{eq:setL},
\[ x_{2l} < x_{2l+1} < \frac{6}{5}i^2. \]
In any case, there exist at least $i^2/200$ terms of the original sequence $(x_k)$ such that 
\[ x_k - \lfloor x_k\rfloor \geq 1/3 \quad\text{and}\quad x_k < \frac{6}{5}i^2. \]
For every such index $k$, the length of the interval \eqref{eq:twotermR} is
\[ \frac{1}{\lfloor x_k\rfloor} - \frac{1}{x_k} 
\geq \frac{x_k - \lfloor x_k\rfloor}{x_k^2} > \frac{25}{108i^4}, \]
and the total length of those intervals is greater than
\[ \frac{i^2}{200} \cdot \frac{25}{108i^4} > \frac{1}{1000(i-1)i}. \]
The last number is $1\text{\textperthousand}$ of the length of the considered interval $(1/i,1/(i-1)]$ and we are done.


\section{Proof of Theorem \ref{thm:main}}
Let
\[ H_n := \sum_{k=1}^{n} \frac{1}{k} \]
denote the $n$-th harmonic number.
We begin by listing some obvious properties of the Egyptian underapproximations.
\begin{enumerate}[label=(P\arabic*), ref=P\arabic*]
\item \label{it:P1} If $q\in\mathbb{Q}\cap[0,\infty)$ is the best $n$-term underapproximation of a number $x\in(0,\infty)$, then it is also the best $n$-term underapproximation of every real from the interval $(q,x]$. 
\item \label{it:P2} If two positive numbers have equal best $n$-term underapproximations, then they also have equal best $(n-1)$-term underapproximations. 
\end{enumerate}
For a positive integer $n$ define an equivalence relation $\sim_n$ on the set $(0,\infty)$ by proclaiming that $x\sim_n y$ if $x$ and $y$ have the same best $n$-term underapproximation. Then \eqref{it:P1} implies the following.
\begin{enumerate}[resume*]
\item \label{it:P3} Every equivalence class of $\sim_n$ is either an interval of the form $(q,r]$ for some rational numbers $q<r$, or the unbounded interval $(H_n,\infty)$. 
\end{enumerate}
Let $\mathcal{I}_n$ denote the collection of all intervals mentioned in \eqref{it:P3}. By definition they make a countable partition of $(0,\infty)$.
Next, \eqref{it:P2} implies the following further property.
\begin{enumerate}[resume*]
\item \label{it:P4} Partition $\mathcal{I}_{n}$ is a refinement of $\mathcal{I}_{n-1}$, i.e., every interval from $\mathcal{I}_{n-1}$ is a countable disjoint union of some intervals from $\mathcal{I}_{n}$.
\end{enumerate}
Finally, we will prove the following.
\begin{enumerate}[resume*]
\item \label{it:P5} Every bounded interval from $\mathcal{I}_{n}$ has length at most $1/(n(n+1))$.
\end{enumerate}
Namely, consider only numbers of the form
\begin{equation}\label{eq:regularnum}
\sum_{k=1}^{n} \frac{1}{m_k} 
= \sum_{k=1}^{l} \frac{1}{k} + \sum_{k=l+1}^{n} \frac{1}{m_k},
\end{equation}
where $m_1<m_2<\cdots<m_n$ are positive integers and there exists $0\leq l< n$ such that $m_k=k$ for $k=1,\ldots,l$ and $m_{k+1}\geq (m_k-1)m_k +1$ for $k=l+1,\ldots,n-1$.
By an easy induction on $n$ we see that each of the numbers \eqref{eq:regularnum}, other than $H_n$, is precisely $1/((m_k-1)m_k)$ apart from the numbers \eqref{eq:regularnum} to its right.
Since we have $m_n\geq n+1$ for every number \eqref{eq:regularnum} different from $H_n$, we conclude that the above numbers are $1/(n(n+1))$-dense in $(0,H_n]$.
No interval $I\in\mathcal{I}_n$ can have a number of the form \eqref{eq:regularnum} in its interior, so each bounded interval from $\mathcal{I}_n$ has length at most $1/(n(n+1))$.
This verifies \eqref{it:P5}.

\smallskip
For integers $0\leq s<t$, let $X_{s,t}$ denote the set of all numbers $x\in(0,H_s]$ for which there exist positive integers $m_1<m_2<\cdots<m_t$ such that $\sum_{k=1}^{n} 1/m_k$ is the best $n$-term Egyptian underapproximation of $x$ for $n=s,s+1,\ldots,t-1,t$.
Clearly, the set of all positive reals with eventually greedy best Egyptian underapproximations is contained in
\begin{equation}\label{eq:keyset}
\bigcup_{s=0}^{\infty} \underbrace{\bigcap_{t=s+1}^{\infty} X_{s,t}}_{\text{increases in }s}. 
\end{equation}
We will simply write $|A|$ for the Lebesgue measure of a measurable set $A\subseteq\mathbb{R}$.

\begin{lemma}\label{lm:inductive}
The estimate
\begin{equation}\label{eq:decayofY}
|X_{s,t+2}| \leq \frac{1999}{2000} |X_{s,t}|
\end{equation}
holds for all integers $100\leq s<t$.
\end{lemma}

\begin{proof}
Each set $X_{s,t}$ is defined in terms of the best $n$-term underapproximations for $n\leq t$. 
(In particular, $(0,H_s]$ is the set of positive reals for which the best $s$-term underapproximation is less than $H_s$.)
By \eqref{it:P2} and \eqref{it:P4}, the set $X_{s,t}$ is a union of intervals from a subcollection of $\mathcal{I}_t$, which was introduced after property \eqref{it:P3}. 
We denote this subcollection by $\mathcal{X}_{s,t}$:
\begin{equation}\label{eq:bigunion}
X_{s,t} = \bigcup_{I\in\mathcal{X}_{s,t}} I.
\end{equation}
Since $X_{s,t}$ is bounded, each interval in $\mathcal{X}_{s,t}$ is bounded too.

Take an arbitrary non-empty interval $I\in\mathcal{X}_{s,t}$ and write it as $I=(q,r]$. 
Let $i_0$ be the smallest positive integer such that $1/i_0< r-q$. 
From \eqref{it:P5} and $t>100$ we know that $|I|<10^{-4}$, so
\[ i_0 > |I|^{-1} > 10^4. \]
The interval $I$ is a disjoint union of
\[ I ':= \Big(q+\frac{1}{i_0},r\Big] \]
and
\[ I_i := q + \Big(\frac{1}{i},\frac{1}{i-1}\Big] \quad\text{for $i=i_0+1,i_0+2,\ldots$}. \]
We treat $I'$ as an exceptional interval and we only need a rough bound on its relative length.
Namely, $q+1/(i_0-1)\geq r$ by the definition of $i_0$, and it means that
\begin{equation}\label{eq:exceptional}
|I'| \leq \frac{1}{i_0-1} - \frac{1}{i_0} = \frac{1}{(i_0-1)i_0} \leq \frac{10^{-4}}{i_0} < 10^{-4} |I|.
\end{equation}
Next, for an integer $i>i_0$ and every $x\in X_{s,t+2}\cap I_i$ we can say that:
\begin{itemize}
\item $q$ is the best $t$-term underapproximation of $x$;
\item $q+1/i$ is the best $(t+1)$-term underapproximation of $x$;
\item $q+1/i+1/j$ is the best $(t+2)$-term underapproximation of $x$ for some integer $j\geq (i-1)i+1$ depending on $x$.
\end{itemize}
From this, we see that the number $x-q\in (1/i,1/(i-1)]$ has greedy best two-term Egyptian underapproximation and it is $1/i+1/j$. 
(Namely, if $i'<j'$ were positive integers such that $1/i+1/j<1/i'+1/j'<x-q$, then we would have $i'>i$ and $q+1/i'+1/j'$ would be a better $(t+2)$-term underapproximation of $x$ than $q+1/i+1/j$.)
Lemma \ref{lm:main} implies that the set of such numbers has measure at most $999/1000$ of the length of $(1/i,1/(i-1)]$, so
\[ |X_{s,t+2}\cap I_i| \leq \frac{999}{1000} |I_i|. \]
Summing in $i=i_0+1,i_0+2,\ldots$ and using \eqref{eq:exceptional} we get
\[ |X_{s,t+2}\cap I| \leq \frac{1999}{2000} |I|. \]
Finally, summing in $I\in\mathcal{X}_{s,t}$, recalling \eqref{eq:bigunion}, and observing $X_{s,t+2}\subseteq X_{s,t}$, we obtain \eqref{eq:decayofY}.
\end{proof}

We can now finalize the proof of Theorem \ref{thm:main}.
Starting from $|X_{s,s+1}|\leq H_s$, $|X_{s,s+2}|\leq H_s$ and repeatedly applying Lemma \ref{lm:inductive}, we conclude
\[ |X_{s,t}| \leq \Big(\frac{1999}{2000}\Big)^{(t-s-2)/2} H_s \]
whenever $100\leq s<t$. Letting $t\to\infty$ we obtain
\[ \Big| \bigcap_{t=s+1}^{\infty} X_{s,t} \Big| = 0. \]
Taking the union over $s\geq100$ and using countable subadditivity of the Lebesgue measure, we conclude that the set \eqref{eq:keyset} also has measure zero.




\section*{Acknowledgments}
This work was supported in part by the Croatian Science Foundation under the project HRZZ-IP-2022-10-5116 (FANAP).
The author would like to thank Rudi Mrazovi\'{c} for a useful discussion and Thomas Bloom for founding the website \cite{EP}, where he discovered numerous interesting problems.
The author is also grateful to Richard Green for discussing the present paper and its background in the popular newsletter \emph{A Piece of the Pi: mathematics explained} \cite{RG2024}.


\bibliography{Egyptian_fractions}

\begin{thebibliography}{10}

\bibitem{EP}
Thomas~F. Bloom.
\newblock {Erd\H{o}s} problems.
\newblock \url{https://www.erdosproblems.com/}.
\newblock Accessed: June 7, 2024.

\bibitem{BE22}
Thomas~F. Bloom and Christian Elsholtz.
\newblock Egyptian fractions.
\newblock {\em Nieuw Arch. Wiskd. (5)}, 23(4):237--245, 2022.

\bibitem{Chu23}
H{\`u}ng~Viet Chu.
\newblock A threshold for the best two-term underapproximation by {E}gyptian fractions.
\newblock {\em Indag. Math. (N.S.)}, 35(2):350--375, 2024.
\newblock \href {https://doi.org/10.1016/j.indag.2024.01.006} {\path{doi:10.1016/j.indag.2024.01.006}}.

\bibitem{Con24}
David Conlon, Jacob Fox, Xiaoyu He, Dhruv Mubayi, Huy~Tuan Pham, Andrew Suk, and Jacques Verstra\"{e}te.
\newblock A question of {{E}rd\H{o}s} and {G}raham on {E}gyptian fractions.
\newblock Available at: https://arxiv.org/abs/2404.16016, 2024.

\bibitem{Cur22}
David~R. Curtiss.
\newblock On {K}ellogg's {D}iophantine {P}roblem.
\newblock {\em Amer. Math. Monthly}, 29(10):380--387, 1922.
\newblock \href {https://doi.org/10.2307/2299023} {\path{doi:10.2307/2299023}}.

\bibitem{Duj21}
Andrej Dujella.
\newblock {\em Number theory}.
\newblock \v{S}kolska Knjiga, Zagreb, 2021.

\bibitem{Erd50}
Paul Erd\H{o}s.
\newblock On a {D}iophantine equation.
\newblock {\em Mat. Lapok}, 1:192--210, 1950.

\bibitem{EG80}
Paul Erd\H{o}s and Ronald~L. Graham.
\newblock {\em Old and new problems and results in combinatorial number theory}, volume~28 of {\em Monographies de L'Enseignement Math\'{e}matique}.
\newblock Universit\'{e} de Gen\`eve, L'Enseignement Math\'{e}matique, Geneva, 1980.

\bibitem{Gra13}
Ronald~L. Graham.
\newblock Paul {Erd\H{o}s} and {E}gyptian fractions.
\newblock In {\em Erd\H{o}s centennial}, volume~25 of {\em Bolyai Soc. Math. Stud.}, pages 289--309. J\'{a}nos Bolyai Math. Soc., Budapest, 2013.
\newblock \href {https://doi.org/10.1007/978-3-642-39286-3\_9} {\path{doi:10.1007/978-3-642-39286-3\_9}}.

\bibitem{RG2024}
Richard Green.
\newblock A {P}iece of the {P}i: mathematics explained. {E}gyptian fractions.
\newblock \url{https://apieceofthepi.substack.com/p/egyptian-fractions}.
\newblock Accessed: September 24, 2024.

\bibitem{KN74}
Lauwerens Kuipers and Harald Niederreiter.
\newblock {\em Uniform distribution of sequences}.
\newblock Pure and Applied Mathematics. Wiley-Interscience [John Wiley \& Sons], New York-London-Sydney, 1974.

\bibitem{LS24}
Yang~P. Liu and Mehtaab Sawhney.
\newblock On further questions regarding unit fractions.
\newblock Available at: https://arxiv.org/abs/2404.07113, 2024.

\bibitem{Nat23}
Melvyn~B. Nathanson.
\newblock Underapproximation by {E}gyptian fractions.
\newblock {\em J. Number Theory}, 242:208--234, 2023.
\newblock \href {https://doi.org/10.1016/j.jnt.2022.07.005} {\path{doi:10.1016/j.jnt.2022.07.005}}.

\bibitem{Sal47}
Herbert~E. Salzer.
\newblock The approximation of numbers as sums of reciprocals.
\newblock {\em Amer. Math. Monthly}, 54:135--142, 1947.
\newblock \href {https://doi.org/10.2307/2305906} {\path{doi:10.2307/2305906}}.

\bibitem{Sou05}
Kannan Soundararajan.
\newblock Approximating $1$ from below using $n$ {E}gyptian fractions.
\newblock Available at: https://arxiv.org/abs/math/0502247, 2005.

\bibitem{Str30}
Gottfried Stratemeyer.
\newblock Stammbruchentwickelungen f\"{u}r die {Q}uadratwurzel aus einer rationalen {Z}ahl.
\newblock {\em Math. Z.}, 31(1):767--768, 1930.
\newblock \href {https://doi.org/10.1007/BF01246446} {\path{doi:10.1007/BF01246446}}.

\bibitem{Syl80}
James~J. Sylvester.
\newblock On a point in the theory of vulgar fractions.
\newblock {\em Amer. J. Math.}, 3(4):332--335, 1880.
\newblock \href {https://doi.org/10.2307/2369261} {\path{doi:10.2307/2369261}}.

\bibitem{Tak21}
Tanzo Takenouchi.
\newblock On an indeterminate equation.
\newblock {\em Proceedings of the Physico-Mathematical Society of Japan}, 3:78--92, 1921.

\end{thebibliography}
\bibliographystyle{plainurl}

\end{document}